\documentclass[a4paper,12pt,twoside]{article}
\usepackage{amsthm,amssymb,amsmath,amscd,url}
\usepackage[all]{xy}
\usepackage{pb-diagram}
\newtheorem{thm}{Theorem}[section]
\newtheorem{conj}[thm]{Conjecture}

\newtheorem{lem}[thm]{Lemma}

\newtheorem{defn}[thm]{Definition}

\newcommand{\Aset}{\mathbb{A}}

\newcommand{\Zset}{\mathbb{Z}}
\newcommand{\Rset}{\mathbb{R}}
\newcommand{\Cset}{\mathbb{C}}

\def\orb{{\mathcal O}}

\newcommand{\bB}{\mathbf{B}}

\newcommand{\ab}{{\rm ab}}

\def\PPhi{{\Phi}}

\def\rG{{{\rm G}}}

\def\hlambda{{{\widehat\lambda}}}

\def\cC{{\mathcal C}}

\def\Nor{{\rm N}}

\def\C{{\rm C}}

\def\ie{{\it i.e.,\,}}

\def\cI{{\mathcal I}}

\def\integers{{\mathfrak o}}

\def\RBC{{\rm BC}}
\def\HP{{\rm HP}}
\def\SL{{\rm SL}}

\def\SO{{\rm SO}}
\def\GL{{\rm GL}}

\def\Hom{{\rm Hom}}

\def\conj{{\rm conj}}

\def\Ind{{\rm Ind}}
\def\Lie{{\rm Lie}}

\def\der{{\rm der}}

\def\Nor{{\rm N}}
\def\orb{{\mathcal O}}
\def\Irr{{\rm Irr}}

\def\Lie{{\rm Lie}}

\def\Prim{{\rm Prim}}

\def\cA{{\mathcal A}}
\def\cR{{\mathcal R}}
\def\cG{{\mathcal G}}

\def\cH{{\mathcal H}}

\def\cO{{\mathcal O}}

\def\cT{{\mathcal T}}

\def\cW{{\mathcal W}}

\def\fb{{\mathfrak b}}
\def\fc{{\mathfrak c}}
\def\fg{{\mathfrak g}}

\def\fs{{\mathfrak s}}
\def\ft{{\mathfrak t}}

\def\fB{{\mathfrak B}}
\def\fC{{\mathfrak C}}

\def\fii{{\mathfrak i}}

\def\fn{{\mathfrak n}}
\def\fG{{\mathfrak G}}

\def\fP{{\mathfrak P}}
\def\fQ{{\mathfrak Q}}
\def\integers{{\mathfrak o}}

\def\fB{{\mathfrak B}}

\def\St{{\rm St}}

\def\tH{{\widetilde H}}

\def\tu{{\tilde u}}

\def\q{{/\!/}}

\def\sgn{{\rm sgn}}
\def\pphi{{\phi}}
\def\PPhi_F{{\rm Frob}}
\def\Frob{{\rm Frob}}
\def\Flag{{\rm Flag}}

\begin{document}
\title{Kazhdan-Lusztig parameters and extended quotients}
\author{Anne-Marie Aubert, Paul Baum and Roger Plymen}
\date{}
\maketitle

\medskip

\begin{abstract}The Kazhdan-Lusztig parameters are important parameters in the representation theory of $p$-adic groups and  affine Hecke algebras.   We show that the Kazhdan-Lusztig parameters have a definite geometric structure, namely that of the extended quotient $T\q W$ of a complex torus $T$ by a finite Weyl group $W$.   More generally, we show that the corresponding parameters, in the principal series of a reductive $p$-adic group with connected centre,  admit such a geometric structure.  This confirms, in a special case, a recent geometric conjecture in \cite{ABP}. 

In the course of this study, we provide a unified framework for Kazhdan-Lusztig parameters on the one hand, and Springer parameters on the other hand.
Our framework contains a  complex parameter $s$, and allows us to \emph{interpolate} between $s = 1$ and $s = \sqrt q$.  When $s = 1$, we recover the  parameters which occur in the Springer correspondence; when $s = \sqrt q$, we recover the Kazhdan-Lusztig parameters.
\end{abstract}

\section{Introduction}   The Kazhdan-Lusztig parameters are important parameters in the representation theory of $p$-adic groups and  affine Hecke algebras.   We show that the Kazhdan-Lusztig parameters have a definite geometric structure, namely that of the extended quotient $T\q W$ of a complex torus $T$ by a finite Weyl group $W$.   More generally, we show that the corresponding parameters, in the principal series of a reductive $p$-adic group with connected centre,  admit such a geometric structure.  This confirms, in a special case, a recent geometric conjecture in \cite{ABP}. 

In the course of this study, we provide a unified framework for Kazhdan-Lusztig parameters on the one hand, and Springer parameters on the other hand.
Our framework contains a  complex parameter $s$, and allows us to \emph{interpolate} between $s = 1$ and $s = \sqrt q$.  When $s = 1$, we recover the  parameters which occur in the Springer correspondence; when $s = \sqrt q$, we recover the Kazhdan-Lusztig parameters, see \S5.   Here, $q = q_F$ is the cardinality of the residue field of the underlying local field $F$.

Let $\mathcal{G}$ denote a reductive split $p$-adic group with connected centre, maximal split torus $\mathcal{T}$.   
Let $G$, $T$ denote the Langlands dual of $\mathcal{G}$, $\mathcal{T}$.   Then the quotient variety $T/W$ plays a central role. For example, we have the  Satake isomorphism
\[
\mathcal{H}(\mathcal{G}, \mathcal{K}) \simeq \mathcal{O}(T/W)
\]
where $\mathcal{O}(T/W)$ denotes the coordinate algebra of  $T/W$, see \cite[2.2.1]{Sh}, and    $\cH(\mathcal{G}, \mathcal{K})$ denotes the algebra (under convolution) of $\mathcal{K}$-bi-invariant functions of compact support on $\mathcal{G}$, where $\mathcal{K} = \mathcal{G}(\mathfrak{o}_F)$ . 
In this article, we will show that the \emph{extended quotient} plays a central role in the context of the Kazhdan-Lusztig parameters.

We will prove that the extended quotient $T\q W$ is a model for the  Kazhdan-Lusztig parameters, see \S4.   More generally, let 
\[
\fs = [\mathcal{T}, \chi]_{\mathcal{G}}
\]
be a point in the Bernstein spectrum of $\mathcal{G}$.   We prove that the extended quotient $T\q W^{\fs}$ attached to $\fs$ is a model of the corresponding parameters attached to $\fs$.  This is our main result, Theorem 4.1.  \emph{The principal series of a reductive $p$-adic group with connected centre has a definite geometric structure. The principal series is a disjoint union:  each component is the  extended quotient of the dual torus $T$ by the  finite Weyl group  $W^{\fs}$ attached to $\fs$.}    This confirms, in a special case,  a recent geometric conjecture in \cite{ABP}.

\bigskip

We also show in \S4 that our bijection is compatible with base change, in the special case of the irreducible smooth representations of  $\GL(n)$  which admit nonzero Iwahori fixed vectors.

\bigskip

The details of our interpolation between Springer parameters and Kazhdan-Lusztig parameters will be given in \S5.   Our formulation creates a projection 
\[
\pi_{\sqrt q} : T\q W \to T/W
\]
which provides a model of the \emph{infinitesimal character}. 

\bigskip

We conclude in \S6 with some carefully chosen examples.
\bigskip

Since  the crossed product algebra $\cO(T)\rtimes W$ is
isomorphic to \[\Cset[X(T)]\rtimes W\,\simeq\,\Cset[X(T)\rtimes W],\]
we obtain a bijection 
\[\Prim\,\Cset[X(T)\rtimes W]\to T\q W\] where $\Prim$ denotes primitive ideals.
By composing this bijection with the bijection  $\mu$  in Theorem 4.1, we finally get a bijection 
\[\Prim\,\Cset[X(T)\rtimes W]\to\fP(G)\]
where $\fP(G)$ denotes the Kazhdan-Lusztig parameters.  
Let $\cI$ be a standard Iwahori subgroup in $\cG$ and let $\cH(\cG,\cI)$
denote the corresponding Iwahori-Hecke algebra, \ie the algebra (for the
convolution product) of compactly
supported $\cI$-biinvariant functions on $\cG$. The algebra is
isomorphic to \[
\cH(X(T)\rtimes W,q)
\] the Hecke algebra of the extended 
affine Weyl group 
$X(T)\rtimes W$, with parameter $q$. The simple modules of
$\cH(\cG,\cI)$ are parametrized by $\fP(G)$ \cite{KL}.
 
Hence  $\fP(G)$ provides a parametrization of the simple modules
of both the Iwahori-Hecke algebra $\cH(X(T)\rtimes W,q)$ and of
the group algebra of $X(T)\rtimes W$ (that is, the algebra 
$\cH(X(T)\rtimes W,1)$).

Note that the existence of a bijection between these sets of simple
modules was already proved by Lusztig (see for instance
\cite[p.~81, assertion~(a)]{LuAst}).
Lusztig's construction needs to pass through the asymptotic Hecke algebra $J$, while
 we have replaced the use of $J$ by the use of the extended
quotient $T\q W$ (which is much simpler to construct).

\section{Extended quotients}  Let $\mathcal{O}(T)$ denote the coordinate algebra of the complex torus $T$.    In noncommutative geometry, one of the elementary, yet fundamental, concepts is that of \emph{noncommutative quotient} \cite[Example 2.5.3]{K}. The \emph{noncommutative quotient} of $T$ by $W$ is the crossed product algebra
\[
\mathcal{O}(T) \rtimes W.
\]
This is a noncommutative unital $\Cset$-algebra.   We need to filter this idea through periodic cyclic homology.  We have an isomorphism
\[
\HP_*(\mathcal{O}(T) \rtimes W) \simeq H^*(T\q W ; \C)
\]
where $\HP_*$ denotes periodic cyclic homology, $H^*$ denotes cohomology, and $T\q W$ is the extended quotient of $T$ by $W$, see \cite{B}.      We recall the definition of the extended quotient $T\q W$. 
\begin{defn} Let 
\[
\widetilde{T} = \{(t,w) \in T \times W :  w \cdot  t = t\}.
\]
The extended quotient
is the quotient
\[
T\q W  : = \widetilde{T}/W
\]
where $W$ acts via $\alpha(t,w) = (\alpha \cdot t, \alpha w \alpha^{-1})$ with $\alpha \in W$. 
\end{defn} 

 Let $W(t)$ denote the   isotropy subgroup of $t$.   Let $\conj (W(t))$ denote the set of conjugacy classes in $W(t)$, and let $[w]$ denote 
 the conjugacy class of $w$ in $W(t)$. The map
\[
\{(t,w)  : t \in T, w \in W(t)\} \to \{(t,c) : t \in T, c \in \conj (W(t))\}
\]
\[
 (t,w) \mapsto (t,[w])
\]
induces a canonical bijection
\[
\{(t,w) : t \in T, w \in W(t)\}/ W \to \{(t,c) :  t \in T, c \in \conj(W(t))\}/ W
\]
where $W$ acts via $\alpha (t,c) = ( \alpha \cdot t, [ \alpha x \alpha^{-1}])$ with $x \in c$.    

Let $\Irr(W(t))$ denote the set of equivalence classes of irreducible representations of $W(t)$.   A choice of bijection between 
$\conj(W(t))$ and $\Irr(W(t))$ then creates a  bijection
\[ 
T\q W \simeq   \{( t, \tau) :  t \in T, \tau \in \Irr(W(t))\}/ W 
  \]  
where $W$ acts via $\alpha (t, \tau) = (\alpha \cdot t, \alpha_*(\tau))$.  
Here, $\alpha_*(\tau)$ is the push-forward of $\tau$ to an irreducible representation of $W(\alpha \cdot t)$. 
  
  This leads us to 
  \begin{defn}
  The extended quotient of the second kind is
  \[ 
(T\q W)_2: =   \{( t, \tau) :  t \in T, \tau \in \Irr(W(t))\}/ W 
  \] 
  \end{defn}
  
  We then have a non-canonical bijection 
 \[
T\q W   \simeq  (T\q W)_2.
\]

Let $T^w$ denote the fixed set $\{t \in T : w \cdot t = t\}$, and let $Z(w)$ denote the centralizer of $w$ in $W$. 
We have 
\begin{align} \label{eqn:(1)}
T\q W = \bigsqcup T^w /Z(w)
\end{align}
where one $w$ is chosen in each conjugacy class in $W$.   Therefore $T\q W$  is a complex affine algebraic variety.   The number of irreducible
components in $T\q W$ is bounded below by $|\conj(W)|$.

The Jacobson topology on the primitive ideal spectrum of $\mathcal{O}(T) \rtimes W$  
induces a topology on $(T\q W)_2$ such that the identity map 
\[
T\q W \to (T\q W)_2
\]
is continuous. From the point of view of noncommutative geometry \cite{K}, the extended quotient of the second kind is a \emph{noncommutative complex affine algebraic variety}.

The transformation groupoid $ T \rtimes W$ is naturally an \'etale groupoid, see \cite[p. 45]{K}.   Its groupoid algebra $\Cset [T \rtimes W]$ is the crossed product algebra
\[
\mathcal{O}(T) \rtimes W .
\]
In the groupoid $T \rtimes W$, we have
\[
\text{source}(t,w) = t, \quad \text{target}(t,w) = w \cdot t
\]
so that the set
\[
\{(t,w) \in T \times W : w \cdot t = t\} 
\]
comprises all the arrows which are \emph{loops}.

The decomposition of the groupoid $T \rtimes W$ into transitive groupoids leads naturally to  Eqn.~(\ref{eqn:(1)}).   
The groupoid  $T \rtimes W$ seems to be a bridge between $T\q W$ and $(T\q W)_2$.

\bigskip

In the context of algebraic geometry, the extended quotient is known as the inertia stack \cite{M}, in which case the notation is
\[
I(T): = \widetilde{T}, \quad \quad  [I(T)/W]: =  T\q W.
\]

\section{The parameters for the principal series}
\label{sec:unram}

Let $\cW_F$ denote the Weil group of $F$, let $I_F$ be the inertia
subgroup of $\cW_F$. Let
$\PPhi_F \subset \cW_F$ denote a geometric Frobenius (a generator of
$\cW_F/I_F \simeq \Zset$). We have $\cW_F/I_F = <\Frob>$. We will think of this as a multiplicative group, with identity element $1$.

Let $\fP(G)$ denote the set of conjugacy classes in
$G$ of pairs $(\Phi,\rho)$ such that $\Phi$ is a morphism
\[
\Phi\colon \cW_F/I_F \times \SL(2,\Cset) \to G\] which is
\emph{admissible}, \ie $\Phi(1, - )$ is a morphism of complex algebraic groups, 
$\Phi(\Frob,1)$ is a semisimple element in
$G$, and $\rho$ is defined in the following way. 

We will adopt the formulation of Reeder \cite{R}.   
Choose a Borel subgroup $B_2$ in $\SL(2,\Cset)$ and let 
 $S_{\Phi} = \Phi(\cW_F \times B_2)$, a solvable subgroup of $G$. 
 Let $\bB^{\Phi}$ denote the variety of Borel subgroups of $G$ containing $S_{\Phi}$. 
 Let $G_{\Phi}$ be the centralizer 
 in $G$ of the image of $\Phi$. Then $G_{\Phi}$ acts naturally on $\bB^{\Phi}$, and hence on the 
 singular homology  $H_*(\bB^{\Phi},\Cset)$. Then $\rho$ is an irreducible representation of $G_{\Phi}$ which appears in the action 
 of $G_{\Phi}$ on $H_*(\bB^{\Phi},\Cset)$.
 
 A Reeder parameter $(\Phi, \rho)$ determines a Kazhdan-Lusztig parameter $(\sigma, u, \rho)$ in the following way. Let 
\[
u_0 =
\left(
 \begin{array}{cc}
1 & 1 \\
0 & 1 
\end{array}
\right) , \quad
T_x =
\left(
\begin{array}{cc}
x & 0\\
0 &  x^{-1}
\end{array}
\right)
\]
and set
\[
u = \Phi(1,u_0), \quad \sigma = \Phi(\Frob, T_{\sqrt q})
\] 
where $q$ is the cardinality of the residue field $k_F$.   Then the triple $(\sigma, u, \rho)$ is a Kazhdan-Lusztig parameter. 
Since $\Phi$ is a homomorphism and 
\[
T_{\sqrt q} \, u_0 \, T_{\sqrt q}^{-1} = \left(
\begin{array}{cc}
1 & q\\
0 &  1
\end{array}
\right) = u_0^q
\]
it follows that 
\[
\sigma u \sigma^{-1} = u^q.
\]

It is worth noting that the set $\fP(G)$ is $q$-independent.

\bigskip

We now move on to the rest of the principal series.   We recall that  
 $\mathcal{G}$ denotes a reductive split $p$-adic group \emph{with 
connected centre}, maximal split torus $\mathcal{T}$, and
$G$, $T$ denote the Langlands dual of $\mathcal{G}$, $\mathcal{T}$.
We assume in addition that the residual characteristic of $F$ is not a
torsion prime for $G$.
 
Let $\fQ(G)$ denote the set of conjugacy classes in
$G$ of pairs $(\Phi,\rho)$ such that $\Phi$ is a continuous morphism
\[
\Phi\colon \cW_F\times \SL(2,\Cset) \to G\] which is
rational on $\SL(2,\Cset)$ and such that $\Phi(\cW_F)$ consists of semisimple 
element in
$G$, and $\rho$ is defined in the following way.

Choose a Borel subgroup $B_2$ in $\SL(2,\Cset)$ and let $S_{\Phi} = \Phi(\cW_F \times B_2)$.
Let $\bB^{\Phi}$ denote the variety of Borel subgroups of $G$ containing
$S_{\Phi}$.
The variety $\bB^{\Phi}$ is non-empty if and only if $\Phi$ factors
through the topological abelianization
$\cW_F^{\ab}:=\cW_F/\overline{[\cW_F,\cW_F]}$ of $\cW_F$ (see
\cite[\S~4.2]{R}). We will assume that $\bB^{\Phi}$ is non-empty, and we
will still denote by $\Phi$ the homomorphim
\[
\Phi\colon \cW_F^{\ab}\times \SL(2,\Cset) \to G.\]
Let $I_F^{\ab}$ denote the image of $I_F$ in $\cW_F^{\ab}$.
The choice of Frobenius $\Frob$ determines a splitting
\begin{equation} \label{eqn:splitting}
\cW_F^{\ab}=I_F^{\ab}\times\langle\Frob\rangle.\end{equation} 
Let $G_{\Phi}$ be the centralizer in $G$ of the image of $\Phi$. Then 
$G_{\Phi}$ acts naturally on $\bB^{\Phi}$, and hence on the 
singular homology of $H_*(\bB^{\Phi},\Cset)$. Then $\rho$ is an
irreducible representation of $G_{\Phi}$ which appears in the action 
 of $G_{\Phi}$ on $H_*(\bB^{\Phi},\Cset)$. 

Let $\chi$ be a smooth quasicharacter of $\cT$ and let $\fs = [\cT,\chi]_{\cG}$ be the point in the Bernstein spectrum $\fB(\cG)$ determined by $\chi$.   Let 
\begin{equation} \label{eqn:Ws}
W^{\fs}  = \{w \in W : w\cdot \fs = \fs\}.
\end{equation}
Let $X$ denote the rational co-character group of $\cT$, identified with
the rational character group of $T$. Let $\cT_0$ be the maximal compact
subgroup of $\cT$. By choosing a uniformizer in $F$, we obtain a splitting
$$\cT=\cT_0\times X,$$ according to which
\[\chi = \lambda\otimes t,\]
where $\lambda$ is a character of $\cT_0$, and $t\in T$. Let
$r_F\colon \cW_F^{\ab}\to F^\times$ denote the reciprocity isomorphism of
abelian class field theory, and let
\begin{equation} \label{eqn:hl}
\hlambda\colon I_F^{\ab}\to T\end{equation} 
be the unique homomorphism satisfying
\begin{equation} \label{eqn:dd}
\eta\circ\hlambda=\lambda\circ\eta\circ r_F,\quad \text{for all $\eta\in
X$},\end{equation}
where $\eta$ is viewed as a character of $T$ on the left side and as a
co-character of $\cT$ on the right side of~(\ref{eqn:dd}).

Let $H$ denote the centralizer in $G$ of the image of $\hlambda$:
\begin{equation} \label{eqn:H}
H=G_{\hlambda}.\end{equation}
The assumption that $G$ has simply-connected derived group implies that
the group $H$ is connected (see \cite[p.~396]{Roc}). 
Note that $H$ itself does not
have simply-connected derived group in general (for instance, if $G$ is
the exceptional group of type $\rG_2$, and $\sigma$ is the tensor square
of a ramified quadratic character of $F^\times$ then $H=\SO(4,\Cset)$).
 
Let $\fQ(G)_{\hlambda}$ be the subset of $\fQ(G)$ consisting of
the $G$-conjugacy classes of all the pairs $(\Phi,\rho)$ such that
$\Phi$ factors through $\cW_F^{\ab}$ and
\[\Phi|_{I_F^{\ab}}=\hlambda.\]
The group $W^\fs$ defined in~(\ref{eqn:Ws}) is a Weyl group: it is the Weyl 
group of $H$ 
(indeed, in the decomposition of \cite[Lemma~8.1~(i)]{Roc} the group 
$C_\chi$ is trivial as proven on \cite[p.~396]{Roc}):
\[W^\fs=W_H.\]

\section{Main result}

\begin{thm} \label{thm:ps} There is a canonical bijection of the extended quotient of the second kind  $(T\q W^\fs)_2$ onto the set
$\fQ(G)_{\hlambda}$ of conjugacy classes of Reeder parameters attached to the point $\fs$ in the Bernstein spectrum of $\mathcal{G}$.  It follows that there is a bijection
\[
\mu^{\fs} : T\q W^{\fs} \simeq \fQ(G)_{\hlambda}
\]
so that the extended quotient $T\q W^{\fs}$ is a model for the Reeder parameters attached to the point $\fs$. 
\end{thm}

The proof of this theorem requires a series of Lemmas.    We recall that 
\[
W^{\fs} = W_H.
\]
The plan of our proof is to begin with an element in the extended quotient of the second kind $(T\q W_H)_2$. 
   Lemmas  4.2 and 4.3 allow us to infer that $W_H(t)$ is a semidirect product
$W_{\fG(t)}\rtimes A_H(t)$.    We now combine the Springer correspondence for $W_{\fG(t)}$ with  Clifford theory for semidirect products (Clifford theory is a noncommutative version of the Mackey machine).   This creates $4$ parameters 
$(t,x,\varrho, \psi)$.  
With this data, and the character $\lambda$ determined by the point $
\fs$, we construct a Reeder parameter $(\Phi, \rho)$ such that
$\Phi(\Frob,1)=t$, $\Phi(1,u_0)=\exp x$ and
the restriction of $\rho$ contains $\varrho$.

\begin{lem} \label{lem:disconnected}
Let $M$ be a reductive algebraic group. Let $M^0$ denote the connected
component of the identity in $M$. Let $T$ be a maximal torus of $M^0$ and
let $B$ be a Borel subgroup of $M^0$ containing $T$. Let 
\[W_{M^0}(T):=\Nor_{M^0}(T)/T\]
denote the Weyl group of $M^0$ with respect to $T$. We set
\[
W_M(T):=\Nor_M(T)/T.\]
\begin{enumerate}
\item[{\rm (1)}] 
The group $W_M(T)$ has the semidirect product decomposition:
\[W_M(T)=W_{M^0}(T)\rtimes (\Nor_M(T,B)/T),\]
where $\Nor_M(T,B)$ denotes the normalizer in $M$ of the pair $(T,B)$.
\item[{\rm (2)}]
We have
\[\Nor_M(T,B)/T\simeq M/M^0=\pi_0(M).\] 
\end{enumerate}
\end{lem}
\begin{proof}
The group $W_{M^0}(T)$ is a normal subgroup of $W_M(T)$. Indeed,
let $n\in\Nor_{M^0}(T)$ and let $n'\in\Nor_M(T)$, then $n'nn^{\prime-1}$
belongs to $M^0$ (since the latter is normal in $M$) and normalizes $T$, that is,
$n'nn^{\prime-1}\in\Nor_{M^0}(T)$. On the other hand,
$n'(nT)n^{\prime-1}=n'nn^{\prime-1}(n'Tn^{\prime-1})=n'nn^{\prime-1}T$.

Let $w\in W_M(T)$. Then $wBw^{-1}$ is a Borel subgroup of $M^0$
(since, by definition, the Borel subgroups of an algebraic group are 
the maximal closed connected solvable subgroups). Moreover, $wBw^{-1}$  
contains $T$. 
In a connected reductive algebraic group, the intersection of two Borel 
subgroups always contains a maximal torus and the two Borel subgroups are 
conjugate by a element of the normalizer of that torus. Hence $B$ and
$wBw^{-1}$ are conjugate by an element $w_1$ of $W_{M^0}(T)$.
It follows that $w_1^{-1}w$ normalises $B$. Hence
\[w_1^{-1}w\in W_M(T)\cap \Nor_{M}(B)=\Nor_{M}(T,B)/T,\] 
that is, \[W_M(T)=W_{M^0}(T)\cdot(\Nor_M(T,B)/T).\] 
Finally, we have
\[W_{M^0}(T)\cap(\Nor_M(T,B)/T)=\Nor_{M^0}(T,B)/T=\{1\},\] 
since $\Nor_{M^0}(B)=B$ and $B\cap \Nor_{M^0}(T)=T$. This proves (1).

We will now prove (2). We consider the following map:
\[\Nor_{M}(T,B)/T\to M/M^0\quad\quad mT\mapsto mM^0.\leqno{(*)}\]
It is injective. Indeed, let $m,m'\in\Nor_{M}(T,B)$ such that
$mM^0=m'M^0$. Then $m^{-1}m'\in M^0\cap\Nor_{M}(T,B)=\Nor_{M^0}(T,B)=T$
(as we have seen above). Hence $mT=m'T$.

On the other hand, let $m$ be an element in $M$. Then $m^{-1}Bm$ is a
Borel subgroup of $M^0$, hence there exists $m_1\in M^0$ such that
$m^{-1}Bm=m_1^{-1}Bm_1$. It follows that $m_1m^{-1}\in\Nor_M(B)$. Also
$m_1m^{-1}Tmm_1^{-1}$ is a torus of $M^0$ which is contained in 
$m_1m^{-1}Bmm_1^{-1}=B$. Hence $T$ and $m_1m^{-1}Tmm_1^{-1}$ are conjugate
in $B$: there is $b\in B$ such that $m_1m^{-1}Tmm_1^{-1}=b^{-1}Tb$. Then 
$n:=bm_1m^{-1}\in\Nor_M(T,B)$. It gives $m=n^{-1}bm_1$. Since $bm_1\in
M^0$, we obtain $mM^0=n^{-1}M^0$. Hence the map $(*)$ is surjective.
\end{proof}

In order to approach the notation in \cite[p.471]{CG}, we let $\fG(t)$ denote the identity component of the centralizer $C_H(t)$:
\[
\fG(t): = C_H^0(t).
\]
Let $W_{\fG(t)}$ denote the Weyl group of $\fG(t)$.

\begin{lem} \label{lem:centrals} 
Let $t \in T$. 
The isotropy subgroup $W_H(t)$ is the group of $\Nor_{C_H(t)}(T)/T$, and we
have
\[W_H(t) = W_{\fG(t)}\rtimes A_H(t)\quad\text{with $A_H(t):=\pi_0(C_H(t))$.}\]
In the case when $H$ has simply-connected derived group, the
group $C_H(t)$ is connected and $W_H(t)$ is then the Weyl group of
$C_H(t)=\fG(t)$.
\end{lem}
\begin{proof} Let $t \in T$. Note that 
 \begin{align*}
 W_H(t) &  = \{w \in W_H : w\cdot t = t\}\\
 & = \{w \in W_H : wtw^{-1} = t\}\\
 & = \{w \in W_H : wt = tw\}\\
 & = W \cap C_H(t).
\end{align*}
Note that $H$ and $C_H(t)$ have a common maximal torus $T$. Now
\begin{align*}
W_H \cap C_H(t)  & = \Nor_H(T)/T \cap C_H(t)\\
& = \Nor_{C_H(t)}(T)/T\\
& = W_{C_H(t)}(T).
\end{align*}
The result follows by applying Lemma~\ref{lem:disconnected} with $M=C_H(t)$.
 
If $H$ has simply-connected derived group, then the
 centralizer $C_H(t)$ is connected by Steinberg's theorem
\cite[\S 8.8.7]{CG}. 
\end{proof}

Let $\tau$ be an irreducible representation of $W_{\fG(t)}$. 
Now we apply the Springer correspondence
to  $\tau$. Note: the Springer correspondence that we are considering
here coincides with that constructed by Springer for a reductive group
over a field of positive characteristic and is obtained
from the correspondence constructed by Lusztig by tensoring the latter by
the sign representation  of $W_{\fG(t)}$ (see \cite{Hot}).

Let $\mathfrak{c}(t)$ denote the Lie algebra of $\fG(t)$, for $x\in\fc(t)$, let $Z_{\fG(t)}(x)$
denote the centralizer of $x$ in $\fG(t)$, via the adjoint representation of 
$\fG(t)$ on $\mathfrak{c}(t)$, and let 
\begin{align}
A_x = \pi_0 (Z_{\fG(t)}(x))
\end{align}
Let $\bB_x$ denote the variety of Borel
subalgebras of $\mathfrak{c}(t)$ that contain $x$.

All the irreducible components of $\bB_x$ have the same dimension $d(x)$
over $\Rset$, see \cite[Corollary 3.3.24]{CG}.   
The finite group $A_x$  acts on the  set
of irreducible components of $\bB_x$ \cite[p. 161]{CG}.

\begin{defn} If a group $A$ acts on the variety $\mathbf{X}$, let
$\cR(A,\mathbf{X})$ denote the  set of irreducible representations of $A$
appearing
in the homology $H_*(\mathbf{X})$, as in \cite[p.118]{R}.  Let
$\cR_{top}(A, \mathbf{X})$ denote the set of irreducible
representations of $A$ appearing in the top homology of $\mathbf{X}$.
\end{defn}

The Springer correspondence yields a one-to-one correspondence
\begin{equation} \label{eqn:Springercor}
(x,\varrho)\mapsto \tau(x,\varrho)\end{equation}
between the set of $\fG(t)$-conjugacy classes of pairs $(x,\varrho)$ formed by a
nilpotent element $x \in \mathfrak{c}(t)$ and an irreducible representation 
$\varrho$ of $A=A_x$ which occurs in $H_{d(x)}(\bB_x, \Cset)$ (that is,
$\varrho\in\cR_{top}(A_x,\bB_x)$) and the set of isomorphism classes of irreducible  
representations of the Weyl group $W_{\fG(t)}$.

\smallskip

We now work with the Jacobson-Morozov theorem \cite[p. 183]{CG}.  Let
$e_0$ be the standard nilpotent matrix in $\mathfrak{sl}(2,\Cset)$:
\[e_0 = \left(
\begin{array}{cc}
0 & 1 \\
0 & 0 \end{array}\right) \]
There exists a rational homomorphism $\gamma : \SL(2, \Cset) \to \fC(t)$
such that its differential $\mathfrak{sl}(2,\Cset) \to \mathfrak{c}(t)$
sends $e_0$ to $x$, see \cite[\S 3.7.4]{CG}.

Define
\begin{eqnarray} \label{eqn:Phi}
\Phi \colon \cW_F^{\ab}\times \SL(2,\Cset) \to G, \quad\quad  (w,\Frob,Y) 
\mapsto \hlambda(w)\cdot t \cdot \gamma(Y)
\end{eqnarray}
\begin{eqnarray} \label{eqn:Upsilon}
\Upsilon \colon \cW_F^{\ab}\times \SL(2,\Cset) \to H, \quad\quad  (w,\Frob,Y) 
\mapsto \hlambda(w)\cdot t \cdot \gamma(Y)
\end{eqnarray}
\begin{eqnarray} \label{eqn:Psi}
\Psi \colon \cW_F^{\ab} \times \SL(2,\Cset) \to \fG(t), \quad\quad 
(w,\Frob,Y) \mapsto \hlambda(w)\cdot t \cdot \gamma(Y)
\end{eqnarray}
\begin{eqnarray} \label{eqn:Xi}
\Xi \colon  \cW_F^{\ab} \times \SL(2,\Cset) \to \fG(t), \quad \quad (w,\Frob,Y) 
\mapsto \hlambda(w)\cdot \gamma(Y).
\end{eqnarray}
where $w$ is any element in $I_F^{\ab}$.

Note that $im\,\Phi\subset H$ (see \cite[\S~4.2]{R}) and that
$C(im \, \Psi) = C(im \, \Upsilon)$, for any element in $C(im \, \Upsilon)$
must commute with $\Upsilon(\Frob) = t$. We also have $C(im \, \Xi) = C(im \,
\Psi) \subset \fC(t)$. Let
\[
A_{\Psi} = \pi_0(C(im \, \Psi)),
\quad \quad A_{\Xi} = \pi_0(C(im \, \Xi)).
\]

\begin{lem} \label{lem:AAA}
We have
\[
A_x = A_{\Xi} = A_{\Psi}.
\]
\end{lem}

\begin{proof}   According to \cite[\S  3.7.23]{CG},  we have
\[
Z_{\fG(t)}(x)  = C(im \, \Xi)\cdot U
\] 
with $U$ the unipotent radical of $Z_{\fG(t)}(x)$. Now $U$ is contractible
via the map
\[
[0,1] \times U \to U, \quad \quad (\lambda, \exp Y) \mapsto \exp( \lambda
Y)
\]
for all $Y \in \fn$ with $\exp \fn = U$.
\end{proof}

Lemma~\ref{lem:AAA} allows us to define
\[
A:  = A_x = A_{\Psi}= A_{\Xi}.
\]

\medskip

Let $\cC(t)$ denote a {\it predual} of $\fG(t)$, \ie
$\fG(t)$ is the Langlands dual of $\cC(t)$. 
Let $\bB^{\Psi}$ (resp. $\bB^{\Xi}$) denote the variety of the Borel
subgroups of $\fG(t)$
which contain $S_{\Psi}: = \Psi(\cW_F\times B_2)$ (resp. $S_{\Xi}: =
\Xi(\cW_F \times B_2) = \gamma(B_2)$).

\begin{lem} \label{lem:bije}
We have
\[
\cR_{top}(A, \bB_x) = \cR(A, \bB^{\Xi}).
\]
\end{lem}
\begin{proof}   Let, as before,  $\tau$ be an irreducible representation
of $W_{\fG(t)}$.   Let $(x,\varrho)$ be the Springer parameter attached to
$\tau$ by the inverse bijection of (\ref{eqn:Springercor}).
Define $\Xi$ as in Eqn.\ref{eqn:Xi}.  Note that $\Xi$ depends on the
morphism $\gamma$, which in turn depends on the nilpotent element $x \in
\mathfrak{c}(t)$.

Then $\Xi$ is a real tempered $L$-parameter for the $p$-adic group
$\cC(t)$, see \cite[3.18]{BM}.  According to several sources, see 
\cite[\S 10.13]{Lu}, \cite{BM}, there is a bijection between
 Springer parameters and Reeder parameters:
\begin{equation}  \label{eqnarray:bij}
(d \gamma(e_0)),\varrho) \mapsto (\Xi, \varrho).
\end{equation}
Now $\varrho$ is an irreducible representation of $A$ which appears
simultaneously in $H_{d(x)}(\bB_x, \Cset)$ and $H_*(\bB^{\Xi}, \Cset)$.
\end{proof}

\medskip

We will recall below a result of Ram and Ramagge, which is based on
Clifford theoretic results developed by MacDonald and Green.

Let $\cH$ be a finite dimensional $\Cset$-algebra and let $\cA$ be a finite group
acting by automorphisms on $\cH$. If $V$ is a finite dimensional module
for
$\cH$ and $a\in \cA$, let ${}^aV$ denote the $\cH$-module with the
action $f\cdot v:=a^{-1}(f)v$, $f\in\cH$ and $v\in V$. Then $V$ is
simple if and only if ${}^aV$ is. Let $V$ be a simple
$\cH$-module. Define the inertia subgroup of $V$ to be
\[\cA_V:=\left\{a\in \cA\;:\;V\simeq {}^aV\right\}.\]
Let $a\in \cA_V$. Since both $V$ and ${}^a V$ are simple, Schur's lemma 
implies that the isomorphism $V\to{}^aV$ is unique up to a scalar multiple. 
For each $a\in\cA_V$ we fix an isomorphism
\[\pphi_a\colon V\to{}^{a^{-1}}V.\]
Then, as operators on $V$,
\[\pphi_av=a(r)\pphi_a,\quad \text{and} \quad
\pphi_a\pphi_{a'}=\eta_V(a,a')^{-1}\pphi_{aa'},\]
where $\eta_V(a,a')\in\Cset^\times$. The resulting function
\[\eta_V\colon \cA_V\times \cA_V\to \Cset^\times,\]
is a cocycle. The isomorphism class of $\eta_V$ is
independent of the choice of the isomorphism $\pphi_a$.

Let $\Cset[\cA_V]_{\eta_V}$ be the algebra with basis $\left\{c_a\,:\,a\in
\cA_V\right\}$ and multiplication given by
\[c_a\cdot c_{a'}=\eta_V(a,a')c_{aa'},\quad\text{for $a,a'\in\cA_V$.}\]

Let $\psi$ be a simple $\Cset[\cA_V]_{\eta_V}$-module. Then putting
\[(fa)\cdot(v\otimes z)=f\pphi_av\otimes c_az,\quad\text{for $f\in\cH$,
$a\in \cA_V$, $v\in V$, $z\in\psi$,}\]
defines an action of $\cH\rtimes \cA_V$ on $V\otimes\psi$.
Define the induced module
\[V\rtimes\psi:=\Ind_{\cH\rtimes \cA_V}^{\cH\rtimes \cA}(V\otimes\psi).\]
\begin{thm} \label{thm:RaRa}
{\rm (Ram-Ramagge, \cite[Theorem~A.6]{RamRam}, Reeder, \cite[(1.5.1)]{R})}
The induced module $V\rtimes\psi$ is a simple $\cH\rtimes \cA$-module,
every simple $\cH\rtimes \cA$-module occurs in this way, and if
$V\rtimes\psi\simeq V'\rtimes\psi'$, then $V$, $V'$ are $\cA$-conjugate, and
$\psi\simeq\psi'$ as $\Cset[\cA_V]_{\eta_V}$-modules.
\end{thm}

One the other hand, it follows from Lemma~\ref{lem:centrals} that the 
isotropy group of $t$ in $W_H$ admits the following semidirect product 
decomposition:
\[W_H(t)=W_{\fG(t)}\rtimes A_H(t)\quad\text{ with
$A_H(t):=\pi_0(C_H(t))$.}\]
Hence the group algebra $\Cset[W_H(t)]$ is a crossed-product algebra 
\[\Cset[W_H(t)]=\Cset[W_{\fG(t)}]\rtimes A_H(t).\]
By applying Theorem~\ref{thm:RaRa} with $\cH=\Cset[W_{\fG(t)}]$ and $\cA=
A_H(t)$, we see that the irreducible representations of $W_H(t)$ are the
\[\tau(x,\varrho)\rtimes\psi,\]
with $\psi$ any simple $\Cset[A_{\tau}]_{\eta_{\tau}}$-module and 
$\tau=\tau(x,\varrho)$.

Let $\cI$ be a standard
Iwahori subgroup in $\cC(t)$, and let $\cH(\cC(t),\cI)$ denote the
corresponding Iwahori-Hecke algebra. Recall that $x=d \gamma(e_0)$.
We will denote by $V=V(x,\varrho)$ the real tempered simple module of
$\cH(\cC(t),\cI)$ which  corresponds to $(x,\varrho)$. Here ``real'' means
that the central character of $V$ is real.

By applying Theorem~\ref{thm:RaRa} with $\cH=\cH(\cC(t),\cI)$ and 
$\cA=A_H(t)$, we obtain the following subset of simple modules for
$\cH(\cC(t),\cI)\rtimes A_H(t)$:
\[V(x,\varrho)\rtimes\psi,\]
with $\psi$ any simple 
$\Cset[A_{V}]_{\eta_V}$-module and $V=V(x,\varrho)$.

\begin{lem} \label{lem:cocycles}
We have \[A_{\tau(x,\varrho)}=A_{V(x,\varrho)}.\]
Moreover, the cocycles $\eta_{\tau(x,\varrho)}$ and $\eta_{V(x,\varrho)}$ can be 
chosen to be equal. 
\end{lem}
\begin{proof} 
Recall that the \emph{closure order on nilpotent adjoint orbits} is defined as
follows
\[ \orb_1\le\orb_2\quad\text{when $\orb_1\subset\overline{\orb_2}$.}\]
\[ \orb_1\le\orb_2\quad\text{when $\orb_1\subset\overline{\orb_2}$.}\]
For $x$ a nilpotent element of $\fc(t)$, we will denote by $\orb_{x}$ the
nilpotent adjoint orbit which contains $x$.
Then as in \cite[(6.5)]{BM}, we define a
\emph{partial order on the representations of $W_{\fG(t)}$} by
\begin{equation} \label{eqn:ordering}
\tau(x_1,\varrho_1)\le\tau(x_2,\varrho_2)\quad\text{when
$\orb_{x_1}\le{\orb}_{x_2}$}.\end{equation}
In this partial order, the trivial representation of $W(t)$ is a minimal
element
and the sign representation of $W(t)$ is a maximal element.

The $W_{\fG(t)}$-structure of $V(x,\varrho)$ is
\begin{equation} \label{eqn:Wstruct}
V(x,\varrho)|_{W_{\fG(t)}}\,=\,\tau(x,\varrho)\,\oplus\,
\bigoplus_{(x_1,\varrho_1)\atop\tau(x,\varrho)<\tau(x,\varrho_1)}
m_{(x_1,\varrho_1)}\,\tau(x_1,\varrho_1),\end{equation}
where the $m_{(x_1,\varrho_1)}$ are non-negative integers. 
(In case $\cC(t)$ has connected centre, (\ref{eqn:Wstruct}) is implied by
\cite[Theorem~6.3~(1)]{BM}, the proof in the general case follows the
same lines.)
In particular, it follows from (\ref{eqn:Wstruct}) that
\begin{equation} \label{eqn:dim}
\dim_{\Cset}\Hom_{W_{\fG(t)}}\left(\tau(x,\varrho),V(x,\varrho)\right)=1.
\end{equation}

Let $a\in A_H(t)$. Since the action of $A_H(t)$ on $W_{\fG(t)}$ comes from
its
action on the root datum, we have (see \cite[2.6.1, 2.7.3]{R}):
\[{}^a\tau(x,\varrho)=\tau(a \cdot x,{}^a\varrho).\] 
Then
\[{}^a V(x,\varrho)|_{W_{\fG(t)}}\,=\,\tau(a \cdot x,{}^a\varrho)\,\oplus\,
\bigoplus_{(x_1,\varrho_1)\atop\tau(x,\varrho)\le\tau(x_1,\varrho_1)}
m_{(x_1,\varrho_1)}\,\tau(a\cdot x,{}^a\varrho_1).\]
Since $\tau(x,\varrho)\le\tau(x_1,\varrho_1)$ if and only if
$\chi(a\cdot x,{}^a\varrho)\le\tau(a\cdot  x_1,{}^a\varrho_1)$, it follows
that ${}^a V(x,\varrho)$ corresponds to the $\fG(t)$-conjugacy class of
$(a\cdot x,{}^a\varrho)$ via the bijection induced by~(\ref{eqnarray:bij}).

Hence \[{}^a V(x,\varrho)\simeq V(x,\varrho)\quad\text{ if and only if }\quad
{}^a\tau(x,\varrho)\simeq\tau(x,\varrho).\] The equality of the inertia
subgroups
\[A_H(t)_{V(x,\varrho)}=A_H(t)_{\tau(x,\varrho)}=:A_H(t)_{x,\varrho}\]
follows. 

Let $\left\{\pphi_a^V\,:\,a \in A_H(t)_{x,\varrho}\right\}$ (resp.
$\left\{\pphi_a^\tau\,:\,a \in A_H(t)_{x,\varrho}\right\}$) a family of
isomorphisms for $V=V(x,\varrho)$ (resp. $\tau=\tau(x,\varrho)$) which determines the
cocycle $\eta_V$ (resp. $\eta_\tau$).
We have
\[\Hom_{W_{\fG(t)}}(\tau,V)\overset{\pphi_a^V}\to
\Hom_{W_{\fG(t)}}(\tau,{}^{a^{-1}}V)\overset{\pphi_a^\tau}\to
\Hom_{W_{\fG(t)}}({}^{{a}^{-1}}\tau,{}^{a^{-1}}V).\]
The composed map is given by a scalar, since
by Eqn.~(\ref{eqn:dim}) these spaces are one-dimensional. We normalize
$\pphi_a^V$ so that this scalar equals to one. This forces $\eta_V$ and
$\eta_\tau$ to be equal.
\end{proof}  

\begin{lem} There is a bijection between Springer parameters
and Reeder parameters for the group $C_H(t)$:
\[(x,\varrho,\psi)\mapsto (\Xi,\varrho,\psi).\]
\end{lem}
\begin{proof}
Lemma~\ref{lem:cocycles} allows us to extend the bijection 
(\ref{eqnarray:bij})  from $\fG(t)$ to $C_H(t)$.
\end{proof}
\begin{lem} We have
\[
\bB^{\Psi} = \bB^{\Xi}.
\]
\end{lem}
\begin{proof}  
We note that
\[
S_{\Psi} = < t > \gamma(B_2), \quad \quad S_{\Xi} = \gamma(B_2)
\]

Let $\fb$ denote a Borel subgroup of the reductive group $C_H(t)$.  Since $\fb$ is maximal among the connected solvable subgroups of $C_H(t)$,  we have 
$\fb \subset \fG(t)$.   Then we have $\fb = T_{\fb}U_{\fb}$ with $T_{\fb}$ a maximal torus in $\fG(t)$, and 
$U_{\fb}$ the unipotent radical of $\fb$.  Note that $T_{\fb} \subset \fG(t)$.  Therefore $yt = ty$ for all $y \in T_{\fb}$. This means that $t$ centralizes $T_{\fb}$, i.e. $t \in Z(T_{\fb})$. In a connected Lie group such as $\fG(t)$, we have 
\[
Z(T_{\fb}) = T_{\fb}\]
 so that $t \in T_{\fb}$. Since $T_{\fb}$ is a group, it follows that $< t > \, \subset T_{\fb}$. 

As a consequence, we have
\[
\fb \supset \,  < t > \gamma(B_2) \iff \fb  \supset \gamma(B_2).
\]
\end{proof}

\medskip

Let $S_{\Upsilon} = \Upsilon(\cW_F \times B_2)$, a solvable subgroup of $H$.
Let $\bB^{\Upsilon}$ denote the variety of Borel subgroups of $H$ containing
$S_{\Upsilon}$.

\begin{lem} We have
\[
\mathcal{R}(A, \bB^{\Upsilon}) = \mathcal{R}(A, \bB^{\Psi})
\]
\end{lem}

\begin{proof}    
We denote the Lie algebra of $\fG(t)$ by $\fg(t)$, and the Lie algebra of $C_H(t)$ by $\fc_H(t)$ so that
\[
\fg(t) = \fc_H(t).
\]
We note that the codomain of $\Psi$ is $\fG(t)$.  

Let $\bB^t$ denote the variety of all Borel subgroups of $G$ which contain $t$.  Let $B \in \bB^t$. 
Then $B \cap \fG(t)$ is a Borel subgroup of $\fG(t)$.    

The proof  in \cite[p.471]{CG} depends on the fact that $\fG(t)$ is connected, and also on
a triangular decomposition of $\Lie(\fG(t))$:
\[
\Lie\,\fG(t) = \fn^t \oplus \ft \oplus \fn_{-}^t
\]
from which it follows that $\Lie\, B \cap \Lie\, \fG(t) = \fn^t \oplus \ft$ is a Borel subalgebra in $\Lie \,\fG(t)$.  The superscript ``$t$'' stands for 
the centralizer of $t$. 

There is a canonical map 
\begin{align} \label{eqn:(7)} 
\bB^t \to \Flag \, \fG(t), \quad B \mapsto B \cap \fG(t)
\end{align}
Now $\fG(t)$ acts by conjugation on $\bB^t$. We have
\begin{align}
\bB^t = \bB_1 \sqcup \bB_2 \sqcup \cdots \sqcup \bB_m
\end{align}
a disjoint union of $\fG(t)$-orbits,  see \cite[Prop. 8.8.7]{CG}. These orbits are the connected components of $\bB^t$, and the irreducible components of the projective variety
$\bB^t$. The above map~(\ref{eqn:(7)}), restricted to any one of these orbits, is a bijection from the $\fG(t)$-orbit onto $\Flag \, \fG(t)$ and is $\fG(t)$-equivariant. It is then clear that 
\[
\bB_j^{\Upsilon} \simeq \Flag \, \fG(t)^{\Psi}
\]
for each $1 \leq j \leq m$.  We also have $t \in S_\Upsilon = S_{\Psi}$.  Now
\[
\bB^{\Upsilon} =  (\bB^t)^{\Upsilon} = (\bB^t)^{\Psi}
\]
and then
\[
H_*(\bB^{\Upsilon}, \Cset) = H_*(\bB_1^{\Psi}, \Cset) \oplus \cdots \oplus H_*(\bB_m^{\Psi}, \Cset)
\]
a direct sum of \emph{equivalent} $A$-modules.
Hence $\varrho$
occurs in $H_*( \bB^{\Upsilon},\Cset)$ if and only if it occurs
$H_*(\bB^{\Psi}, \Cset)$. 
\end{proof} 

\medskip
Recall that $x$ is a nilpotent element in $\fc(t)$ (the Lie algebra of
$\fG(t)$).
Define 
\[A^+:=\pi_0(Z_{C_H(t)}(x)).\]

\begin{lem}
We have
\[\cR(A,\bB^\Upsilon)=\cR(A^+,\bB^\Upsilon).\]
\end{lem}
\begin{proof}
Choose an isogeny $\iota\colon\tH\to H$ with $\tH_\der$ simply connected 
(as in \cite[Theorem~3.5.4]{R}) such that $H=\tH/Z$ where $Z$ is a finite
subgroup of the centre of $\tH$ (see \cite[\S~3]{R}). Let
$\tilde t$ be a lift of $t$ in $\tH$, that is, $\iota(\tilde t)=t$.
Then we have (see \cite[\S~3.1]{R}):
\begin{equation} \label{eqn:iotacent}
\iota(C_{\tH}(\tilde t))=C_H^0(t)=\fG(t).\end{equation} 
Let $u:=\exp(x)$, a unipotent element in $\fG(t)$.  It follows from
Eqn.~(\ref{eqn:iotacent}) that there exists $\tu\in C_{\tH}(\tilde t)$
such that $u=\iota(\tu)$. 
Recall that $A=\pi_0(Z_{\fG(t)}(x))$. Then 
\[A\simeq\pi_0(Z_{\fG(t)}(u))=\pi_0(Z_{\iota(C_{\tH}(\tilde
t))}(\iota(\tu)))\simeq \pi_0(Z_{C_{\tH}}(\tilde t,\tu)),\] 
and $A$ is a subgroup of $\pi_0(Z_{C_H(t)}(u))\simeq A^+$ (see \cite[\S~3.2--3.3]{R}).

Recall from \cite[Lemma~3.5.3]{R} that 
\[(\tilde t,\tu,\varrho,\psi) \mapsto (t, u,
\rho)\] induces a bijection between
$G$-conjugacy classes of quadruples $(\tilde t,\tu,\varrho,\psi)$ and 
$G$-conjugacy classes of triples $(t,u,\rho)$, where $\rho\in\cR(A^+,\bB^\Upsilon)$ is 
such that the restriction of $\rho$
to $A$ contains $\varrho$. 
\end{proof}

\begin{lem}
We have
\[\cR(A^+,\bB^\Upsilon)=\cR(A^+,\bB^\Phi).\]
\end{lem}
\begin{proof}
It follows from \cite[Lemma~4.4.1]{R}.
\end{proof}

The proof can be reversed.   Here is the reason for this claim:  Lemmas 4.5, 4.6, 4.8 4.10 -- 4.13 are all equalities, and Lemma 4.9 is a bijection.

 This creates a canonical bijection between the extended quotient of the second kind $(T\q W^{\fs})_2$ and $\fQ(G)_{\hlambda}$:
 \begin{align}
\mu \colon (T\q W^{\fs})_2 \longrightarrow \fQ(G)_{\hlambda}, \quad \quad (t, x, \varrho, \psi) \mapsto (\Phi, \rho).
\end{align}
This in turn creates a bijection 
 \begin{align}
T\q W^{\fs} \longrightarrow \fQ(G)_{\hlambda}.
\end{align}
This bijection is not canonical in general, depending as it does
on a choice of bijection between the set of conjugacy classes in
$W_H(t)$ and the set of irreducible characters of $W_H(t)$.
When $G = \GL(n)$, the finite group $W_H(t)$ is a product of
symmetric groups: in this case there is a canonical bijection
between the set of conjugacy classes in $W_H(t)$ and the set of
irreducible characters of $W_H(t)$, by the classical theory of Young tableaux.

To close this section, we will consider the case of $\GL(n,F))$, and the Iwahori point $\fii$ in the Bernstein spectrum of $\GL(n,F)$.
The Langlands dual of $\GL(n,F)$ is $\GL(n,\C)$, and we will take $T$ to be the standard maximal torus in $\GL(n,\C)$.   The Weyl group is the symmetric group $S_n$.   We will denote our bijection, in this case canonical, as follows:
\[
\mu_{F}^{\fii} : T\q W \to \fP(\GL(n,F))
\]
Let $E/F$ be a finite Galois extension of the local field $F$.  According to \cite[Theorem 4.3]{MP}, we have a commutative diagram
\[
\begin{CD}
T\q W @> \mu_{F}^{\fii} >>  \fP(\GL(n,F))\\
@ V   VV        @VV \RBC_{E/F} V\\
T\q W @> \mu_{E}^{\fii} >>  \fP(\GL(n,E))
\end{CD}
\]
In this diagram, the right vertical map $\RBC_{E/F}$ is the standard base change map  sending one Reeder parameter to another as follows:
\[
(\Phi,1) \mapsto (\Phi_{|W_E},1).
\]

Let  \[f = f(E,F)\] denote the residue degree of the extension $E/F$.    We proceed to describe the left vertical map.   We note that  the action of W on T is as automorphisms of the algebraic group T.   Since  $T$  is a group, the map \[T \to T, \quad  t \mapsto t^f\]  is well-defined for any positive integer $f$.   The map
\[
\widetilde{T} \to \widetilde{T}, \quad (t,w) \mapsto (t^f,w)
\]
is also well-defined, since
\[
w\cdot t^f = wt^fw^{-1} = wtw^{-1}wtw^{-1} \cdots wtw^{-1} = t^f.
\]
Since
\[
\alpha\cdot(t^f) = (\alpha\cdot t)^f
\]
for all $\alpha \in W$, this induces a map
\[
T\q W \to T\q W
\]
which is an endomorphism (as algebraic variety) of the extended quotient $T\q W$.  We shall refer to this endomorphism as the \emph{base change endomorphism of degree $f$.}  The left vertical map is the base change endomorphism of degree $f$, according to \cite[Theorem 4.3]{MP}.  That is, our bijection $\mu^{\fii}$  is compatible with base change for $\GL(n)$. 

When we restrict our base change endomorphism from the extended quotient $T\q W$ to the ordinary quotient $T/W$, we see that the  commutative diagram 
containing $\RBC_{E/F}$ is consistent with \cite[Lemma 4.2.1]{Haines}.

\section{Interpolation} 

We will now provide details for the interpolation procedure described in \S1.  We will focus on the Iwahori point $\fii \in \fB(\mathcal{G})$, \ie on the smooth irreducible representations of $\mathcal{G}$ which admit nonzero Iwahori fixed vectors.    To simplify notation, we will write $\mu = \mu^{\fii}$.   Let $\fP(G)$ denote the set of conjugacy classes in
$G$ of Kazhdan-Lusztig parameters.   For each $s \in \Cset^{\times}$, we construct a  commutative diagram:
\[
\begin{CD}
T\q W @> \mu >>  \fP(G)\\
@ V \pi_s VV        @VV i_s V\\
T/W @=  T/W
\end{CD}
\]
in which the map $\mu$ is bijective.  In the top row of this diagram, the set $T\q W$, the set $\fP(G)$, and the map $\mu$ are independent of the parameter $s$.  
\medskip

We start by defining the vertical maps $i_s$, $\pi_s$ in the diagram.  Let $s \in \Cset^{\times}$.  We will define
\begin{align}
i_s: \fP(G) \to T/W, \quad  (\Phi, \rho) \mapsto  \Phi (\Frob, T_s)\end{align}
\begin{align}
\pi_s : T\q W \to T/W, \quad (t,w) \mapsto t \cdot \gamma (T_s)
\end{align}
where $(\Phi, \rho)$ is a Reeder parameter, and $(t,w) \in T\q W$. We note that 
\[
 \Phi (\Frob, T_s) = t \cdot \gamma(T_s)
\]
so that the diagram is commutative.

$\bullet$ Let $s = 1$, and assume, for the moment, that $C_H(t)$ is connected.    The map $\mu$ in Theorem 4.1  sends $(t, \tau) $ to $(\Phi ,\rho)$.   We note that
\[
t = \Phi(\Frob, T_1) = \Phi(\Frob,1).\]
The map $\mu$  determines the map
\[
(t, \tau) \mapsto (t, \Phi(1,u_0), \rho)
\]
which, in turn, determines the map
\[
\tau \mapsto (\exp(x), \rho)
\]
which is the Springer correspondence for the Weyl group $W_H(t)$. 

$\bullet$  Now let $s = \sqrt q$ where $q$ is the cardinality of the residue field $k_F$ of $F$.    We now link our result to the representation theory of the $p$-adic group $\mathcal{G}$ as follows.   As in \S 3, let
 \[
\sigma: =  \Phi (\Frob, T_{\sqrt q}), \quad \quad u: = \Phi(1,u_0).
\]
Then we have 
\[
\sigma u \sigma^{-1} = u^q
\]
and the triple $(\sigma, u, \rho)$ is a Kazhdan-Lusztig triple.

The correspondence $\sigma \mapsto \chi_{\sigma}$ between points in $T$ and unramified quasicharacters of $\mathcal{T}$ can be fixed by the relation
\[
\chi_{\sigma}(\lambda(\varpi_F)) = \lambda(\sigma)
\]
where $\varpi_F$ is a uniformizer in $F$, and $\lambda \in X_*(\mathcal{T}) = X^*(T)$.   The Kazhdan-Lusztig triples $(\sigma, u, \rho)$ parametrize the irreducible constituents of the (unitarily) induced representation
\[
\Ind_{\mathcal{B}}^{\mathcal{G}}(\chi_{\sigma}\otimes 1).
\]

Note that
\[
i_{\sqrt q}: (\Phi,\rho) \mapsto \sigma
\]
so that $i_{\sqrt q}$ is the \emph{infinitesimal character}.  The infinitesimal character is denoted $\mathbf{Sc}$ in \cite[VI.7.1.1]{Renard} ($\mathbf{Sc}$ for \emph{support cuspidal})

Since $\mu$ is bijective and the diagram is commutative,  the number of points 
in the fibre of the $q$-projection $\pi_{\sqrt q}$ equals the number of inequivalent irreducible constituents of
$\Ind_{\mathcal{B}}^{\mathcal{G}}(\chi_{\sigma}\otimes 1)$:
\begin{align}
|\pi^{-1}_{\sqrt q}(\sigma)| = |\Ind_{\mathcal{B}}^{\mathcal{G}}(\chi_{\sigma}\otimes 1) |
\end{align}
The $q$-projection $\pi_{\sqrt q}$ is a model of the infinitesimal character $\mathbf{Sc}$. 
\medskip

Our formulation leads to  Eqn.(24), which  appears to have some predictive power.    Note that the definition of the $q$-projection $\pi_{\sqrt q}$ depends only on the $L$-parameter $\Phi$.  An $L$-parameter determines an $L$-packet, and does not determine the number of irreducible constituents of the $L$-packet.

 \section{Examples}

\textsc{Example~1.}  \emph{Realization of the ordinary quotient}  $T/W$.    Consider an $L$-parameter $\Phi$ for which $\Phi | _{\SL(2,\Cset)} = 1$.  Let $t = \Phi(\Frob)$. Then 
\[
G_{\Phi} : = C(im \, \Phi) = C(t)
\]
so that $G_{\Phi}$ is connected and acts trivially in homology. Therefore $\rho$ is the unit representation $1$.

Now $t$ is a semisimple element in $G$, and all such semisimple elements arise.  Modulo conjugacy in $G$, the set of such $L$-parameters $\Phi$ is  parametrized by the quotient $T/W$.  Explicitly, let
\[
\fP_1(G): = \{\Phi \in \fP(G):   \Phi |_{ \SL(2,\Cset)} = 1 \}.
\]  
   Then we have a canonical bijection
\[
\fP_1(G) \to T/W, \quad \quad (\Phi,1) \mapsto \Phi(\Frob,1)
\]
which fits into the commutative diagram
\[
\begin{CD}
\fP_1(G) @>>>T/W \\
@VVV              @VVV   \\
\fP(G)  @>>> T \q W
\end{CD}
\]
where the vertical maps are inclusions. 
\bigskip

\textsc{Example~2.} \emph{The general linear group}.  Let $\cG = \GL(n), G = \GL(n,\Cset)$. Let
\[
\Phi = \chi \otimes \tau(n)
\]
where $\chi$ is an unramified quasicharacter of $\cW_F$ and $\tau(n)$ is the irreducible $n$-dimensional representation of $\SL(2,\Cset)$.    By local classfield theory, the quasicharacter $\chi$ factors through $F^{\times}$.   In the local Langlands 
correspondence for $\GL(n)$, the image of $\Phi$ is the unramified twist $\chi  \circ  \det$ of the Steinberg representation $\St(n)$.  

The sign representation $sgn$ of the Weyl group $W$ has Springer parameters $(\mathcal{O}_{prin},1)$, where $\mathcal{O}_{prin}$ is the principal orbit in $\mathfrak{gl}(n,\Cset)$. In the \emph{canonical} correspondence between irreducible representations of $S_n$ and conjugacy classes in $S_n$, the trivial representation of
$W$ corresponds to the  conjugacy class containing  the $n$-cycle $w_0 = (123 \cdots n)$.

Now $G_{\Phi} = C(im \, \Phi)$ is connected \cite[\S3.6.3]{CG}, and so acts trivially in homology. 
   Therefore $\rho$ is the unit representation $1$.  The image $\Phi(1,u_0)$ is a regular nilpotent, i.e. a nilpotent with one Jordan block (given by the partition of $n$ with one part).   The corresponding conjugacy class in $W$ is $\{w_0\}$.   The corresponding irreducible component  of the extended quotient is
$$T^{w_0}/Z(w_0) = \{(z,z, \ldots,z): z \in \Cset^{\times}\}  \simeq \Cset^{\times}.$$  This is our model, in the extended quotient picture, of the complex $1$-torus of all unramified twists of the Steinberg representation $\St(n)$. The map from $L$-parameters to pairs $(w,t) \in T \q W$ is given by
\[
\chi \otimes \tau(n) \mapsto (w_0, \chi(\Frob), \dots, \chi(\Frob)).
\]
Among these representations, there is one real tempered representation, namely $\St(n)$, with $L$-parameter $1 \otimes \tau(n)$, 
attached to the principal orbit $\cO_{prin} \subset G$.

More generally, let 
\[
\Phi = \chi_1 \otimes \tau(n_1) \oplus \cdots \oplus \chi_k \otimes \tau(n_k)
\]
where $n_1 + \cdots + n_k = n$ is a partition of $n$.   This determines the unipotent orbit $\cO(n_1, \ldots, n_k) \subset G$.  There is a  conjugacy class
in $W$ attached canonically to this orbit: it contains the product of disjoint cycles of lengths $n_1, \ldots, n_k$. The fixed set is a complex torus, and the
component in $T\q W$ is a product of symmetric products of complex $1$- tori.

\bigskip

\textsc{Example 3}.   \emph{The exceptional group of type $\rG_2$}.  This example contains a Reeder parameter $(\Phi,\rho)$  with 
$\rho \neq 1$. 
The torus $\cT$ is identified with $F^\times\times F^\times$. We take 
$\lambda=\chi\otimes\chi$ where $\chi$ is a nontrivial quadratic character of 
$\integers_F^\times$. 

Here we have $H=\SO(4,\Cset)\simeq\SL(2,\Cset)\times\SL(2,\Cset)/\{\pm I\}$. This complex reductive Lie group is neither simply-connected nor of adjoint type.
We have  $W^\fs=W_H=\Zset/2\Zset\times \Zset/2\Zset$. 
We will write
\[\SL(2, \Cset) \times \SL(2, \Cset) \longrightarrow H^\fs, \quad
(x,y) \mapsto [x,y],\]
\[T_{s,s'}=[T_s,T_{s'}],\quad s,s'\in \Cset^\times.\]

We have
\[\fQ(G)_{\hat\lambda}\to T\q W_H\simeq \Aset^1\sqcup \Aset^1\sqcup pt_1\sqcup
pt_2\sqcup pt_*\sqcup T/W_H,\]
where 
\begin{itemize}
\item
one $\Aset^1$ corresponds to $(\Phi,1)$ with $\Phi(\Frob,1)=[I,T_s]$
and $\Phi(1,u_0)=[u_0,I]$, 
\item the other $\Aset^1$ corresponds to $(\Phi,1)$ 
with $\Phi(\Frob,1)=[T_s,I]$
and $\Phi(1,u_0)=[I,u_0]$, 
\item $pt_1$ corresponds to $(\Phi,1)$ 
with $\Phi(\Frob,1)= T_{1,1}$
and $\Phi(1,u_0)=[u_0,u_0]$,
\item $pt_2$ corresponds to $(\Phi,1)$
with $\Phi(\Frob,1)= T_{1,-1}$
and $\Phi(1,u_0)=[u_0,u_0]$, 
\item
$T/W_H$ corresponds to $(\Phi,1)$
with $\Phi(\Frob,1)= T_{s,s'}$ $s,s'\in \Cset^\times$,
and $\Phi(1,u_0)=[I,I]$,
\item
$pt_*$ corresponds to $(\Phi,\sgn)$
with $\Phi(\Frob,1)= T_{i,i}$, $i=\sqrt{ -1}$
and $\Phi(1,u_0)=[I,I]$.
\end{itemize}

\bigskip

\emph{Acknowledgement}. We would like to thank  A. Premet for drawing our attention to reference \cite{CG}.

Anne-Marie Aubert, Institut de Math\'ematiques de Jussieu, U.M.R. 7586 du C.N.R.S., Paris, France\\
Email: aubert@math.jussieu.fr\\
Paul Baum, Pennsylvania State University, Mathematics Department, University Park, PA 16802, USA\\
Email: baum@math.psu.edu\\
Roger Plymen, School of Mathematics, Alan Turing building, Manchester University, Manchester M13 9PL, England\\
Email: plymen@manchester.ac.uk
\end{document}